\documentclass[reqno,a4paper,12pt]{amsart}
 \usepackage{amsmath,amscd,amsfonts,amssymb}
 \usepackage{mathrsfs,dsfont}
 \numberwithin{equation}{section} 

\newcommand{\define}[1]{\emph{#1}}

\def\R{\mathbb{R}}

\def\Z{\mathbb{Z}}

\def\Lam{\Lambda}
\def\1{\mathds{1}}

\renewcommand\le{\leqslant}
\renewcommand\ge{\geqslant}

\renewcommand\geq{\geqslant}

\renewcommand\hat{\widehat}

\newcommand\dotprod[2]{\langle #1 , #2 \rangle}

\theoremstyle{plain}
\newtheorem{theorem}{Theorem}[section]
\newtheorem{lem}[theorem]{Lemma}

\newtheorem*{claim*} {Claim}

\newcommand{\thmref}[1]{Theorem~\ref{#1}}
\newcommand{\secref}[1]{Section~\ref{#1}}

\newcommand{\lemref}[1]{Lemma~\ref{#1}}

\theoremstyle{definition}
\newtheorem*{remark*}{Remark}

\newenvironment{enumerate-math}
{\begin{enumerate}
\addtolength{\itemsep}{5pt}
}
{\end{enumerate}}

\begin{document}

 \title{Fourier frames for singular measures and pure type phenomena}

\author{Nir Lev}
\address{Department of Mathematics, Bar-Ilan University, Ramat-Gan 5290002, Israel}
\email{levnir@math.biu.ac.il}

\date{April 4, 2017}
\thanks{Research supported by ISF grant No.\ 225/13 and ERC Starting Grant No.\ 713927.}

\begin{abstract}
Let $\mu$ be a positive measure on $\R^d$. It is known that if the space $L^2(\mu)$ has a frame of exponentials then the measure $\mu$ must be of ``pure type'': it is either discrete,  absolutely continuous or singular continuous. It has been  conjectured that  a similar phenomenon should be true within the class of singular continuous measures, in the sense that $\mu$ cannot admit an exponential  frame if it has components of different dimensions. We prove that this is not the case by showing that the sum of an arc length measure and a surface measure can have a frame of exponentials. On the other hand we prove that a measure of this form cannot have a frame of exponentials if the surface has a point of non-zero Gaussian curvature. This is in spite of the fact that each ``pure'' component of the measure separately may admit such a frame.
\end{abstract}

\maketitle

% ========================================================================

\section{Introduction}

\subsection{}
Let $\mu$ be a positive and finite Borel measure on $\R^d$.
By a \define{Fourier frame} for the space  $L^2(\mu)$ one means a
system of exponential functions
\[ E(\Lambda)=\left\{ e_\lambda \right\}_{\lambda\in\Lambda}, \quad
e_\lambda(x)=e^{2\pi i \dotprod{\lambda}{x}} \]
(where $\Lambda$ is a countable subset of $\R^d$) which constitutes a
\define{frame} in the space. The latter means that there are constants
$0<A,B<\infty$ such that the inequalities 
\[
		A \|f\|^2_{L^2(\mu)}
		 \le \sum_{\lambda\in\Lambda} \left|
		\dotprod{f}{e_\lambda}_{L^2(\mu)} \right|^2 
	\le B\|f\|^2_{L^2(\mu)} 
\]
hold for every $f\in L^2(\mu)$.
The existence of a Fourier frame $E(\Lambda)$ for $L^2(\mu)$ allows one
to decompose in a ``stable'' (but generally non-unique)
way any function $f$ from the space in a Fourier series
$ f= \sum_{\lambda\in \Lambda} c_\lambda e_\lambda $
with frequencies belonging to $\Lambda$ (see \cite{You01}).
\par
For which measures $\mu$ does 
a Fourier frame  exist? The origin of this question goes back to Fuglede
\cite{Fug74} who initiated a study of orthogonal bases 
of exponentials over domains in $\R^d$ endowed with Lebesgue measure,
and to Jorgensen and Pedersen \cite{JP98} who 
discovered the existence of fractal measures which admit such
orthogonal bases. It is well-known, however, that the existence of an orthogonal
basis $E(\Lambda)$ for $L^2(\mu)$ is a strong requirement, satisfied
by a relatively small class of measures $\mu$. Hence it is of interest
to understand whether measures which do not admit such a basis can at least have
a frame of exponentials. 

\subsection{}
It is known \cite{LW06, HLL13} that if a measure $\mu$ has a Fourier
frame, then it must be of ``pure type'': $\mu$ is either discrete,
absolutely continuous or singular continuous. 
\par
The case when $\mu$ is a
discrete measure is well understood: $\mu$ has a Fourier frame if and
only if it has finitely many atoms \cite{HLL13}. It is also known
precisely which absolutely continuous measures $\mu$ have a Fourier
frame: this is the case if and only if $\mu$ is supported on a set of
finite Lebesgue measure in $\R^d$, and its density function is bounded
from above and from below almost everywhere on the support
\cite{Lai11, DL14, NOU16} (in the last paper the authors used the
solution of the Kadison-Singer problem \cite{MSS15} to prove the
existence of a Fourier frame in the case when the support is
unbounded). 
\par
The case when the measure $\mu$ is singular and
continuous, on the other hand, is much less understood. In this
context mainly self-similar measures have been studied (see e.g.\
\cite{DLW16} and the references therein). However, even the following
question \cite{Str00} is still open: does the middle-third Cantor
measure have a Fourier frame? (It is known \cite{JP98} that this
measure has no orthogonal basis of exponentials.)
\subsection{}
It has been conjectured that a  pure type phenomenon should also
exist within the class of singular continuous measures.
Namely, if such a measure has ``components of different dimensions'' then
it cannot have a Fourier frame. A concrete formulation of such a
conjecture was given explicitly in \cite[Conjecture 5.2]{DL14}. 
\par
The
paper \cite{DHSW11} contains some results in this direction. These
results establish a connection between the fractal dimension of some
self-similar measures $\mu$ and the ``Beurling dimension'' of certain
Fourier frames for these measures. 
\par
However, in the present paper we obtain some results which
contradict the above conjecture. In particular, we prove the following:
\begin{theorem} \label{thm_A2.1}
	There is a measure $\mu$ on $\R^d$ $(d\ge 3)$ which is the
	sum of the arc length measure on a curve and the area measure
	on a hypersurface, such that $L^2(\mu)$ has a Fourier
	frame.
\end{theorem}
The result thus shows that a measure with a Fourier frame can nevertheless have
components of different dimensions. Actually, the measure $\mu$ in our
example is simply the sum of the arc length measure on the segment
$[0,1]\times \{0\}^{d-1}$ and the area measure on the hypersurface
$\{0\}\times  [0,1]^{d-1}$. We will obtain 
\thmref{thm_A2.1} as a special case of a more general
result (\thmref{thm_B2.1}) that allows one to construct many examples
of ``mixed type'' measures which have a Fourier frame.

\subsection{}
On the other hand, we will see that under additional geometric assumptions on the
components of the measure, one can establish pure type phenomena with
respect to the dimension of these components. For example, we will prove the
following result: 
\begin{theorem}
	\label{thm_A2.2}
	Let $\mu$ be a measure on $\R^d$ $(d\ge 3)$ which is the sum
	of the arc length measure on an open subset of a smooth
	curve, and the area measure on an open subset of a smooth
	hypersurface. If the hypersurface has a point of non-zero Gaussian curvature,
	then $L^2(\mu)$ does not admit a Fourier frame.
\end{theorem}
It is instructive to notice the contrast between the curvature
 requirement in this result and the ``flatness'' of the
 hypersurface in the example of \thmref{thm_A2.1}. We will deduce
 \thmref{thm_A2.2}  from a more general result  (\thmref{thm_C3.1}) 
that establishes a certain pure type phenomenon within the class of  singular continuous measures.

\subsection{}
In order to interpret \thmref{thm_A2.2} as a genuine ``pure type principle'', it is
desirable  to verify that the assumptions on the measure $\mu$ in the theorem nevertheless
permit the existence of a Fourier frame for each  ``pure'' component of the measure separately. 
This would mean that the conclusion that $\mu$ admits no Fourier frame is really due to the 
combination of these components together in the measure $\mu$.
\par
That this is indeed the case can be seen from the following result, which provides many examples of
 ``single dimensional'' measures which do admit a Fourier frame: 
 \begin{theorem}
	 \label{thm_A2.3}
	 Let $\mu$ be a measure on $\R^d$, which is the
	 $k$-dimensional area measure on a compact subset of the
	 graph $\{ (x,\varphi(x)) :  x\in \R^k \}$ of a smooth
		 function $\varphi: \R^k\to \R^{d-k}$ $(1\le k\le
		 d-1)$. Then $L^2(\mu)$ has a Fourier frame. 
 \end{theorem}
 This result has a quite straightforward proof (see \secref{sec_D1}), but its main point is to clarify
 that the assumptions in \thmref{thm_A2.2} indeed do not prevent the
 existence of a Fourier frame for each  ``pure'' component
 of the measure separately.

% ========================================================================

\section{Mixed type measures with a Fourier frame}
\label{sec_B1}
In  this section we present a method for constructing examples of
measure which have components of different dimensions, but
nevertheless have a Fourier frame. As a special case we will obtain
\thmref{thm_A2.1}.
\subsection{}
Let $\mu$ be a measure on $\R^n$, and $\nu$ be a measure on $\R^m$. We
assume that both measures are positive and finite on their respective
spaces. From the two measures $\mu, \nu$ one can construct a new
measure $\rho$ on $\R^{n+m}=\R^n\times\R^m$, defined by 
\begin{equation}
	\label{eq_B2.1}
	\rho = \mu\times \delta_0+\delta_0\times \nu,
\end{equation}
where $\delta_0$ denotes the Dirac measure at the origin.
Equivalently, the measure $\rho$ may be defined by the requirement
that 
\[
	\int_{\R^n\times\R^m} f(x,y) d\rho(x,y) =
	\int_{\R^n}f(x,0)d\mu(x)+\int_{\R^m}f(0,y)d\nu(y)
\]
for every continuous, compactly supported function $f$ on $\R^n\times \R^m$.
\par
Notice that $\rho$ is a singular measure, whose support is contained in the union
of the two proper subspaces $\R^n \times \{0\}^m$ and $\{0\}^n \times \R^m$
of  $\R^n\times \R^m$.
\begin{theorem}
	\label{thm_B2.1}
	Assume that $\mu,\nu$ are continuous measures, and that each
	one of them has a Fourier frame. Then also the measure
	$\rho$ given by \eqref{eq_B2.1} has a Fourier frame.
\end{theorem}
For example, if we take $\mu$ (respectively $\nu$) to be the Lebesgue
measure on the segment $[0,1]$ (respectively on the cube
$[0,1]^{d-1}$) then the corresponding measure $\rho$ is the sum of the
arc length measure on the curve $[0,1]\times \{0\}^{d-1}$ and the area
measure on the hypersurface $\{0\}\times [0,1]^{d-1}$. Since in this
case both measures $\mu$ and $\nu$ have a Fourier frame (and,
in fact, even an orthonormal basis of exponentials), it follows from 
\thmref{thm_B2.1} that also $\rho$ has a Fourier
frame. Thus we obtain \thmref{thm_A2.1} above.
\par
 Generally speaking, the dimensions of the
components $\mu\times \delta_0$ and $\delta_0 \times \nu$ of the measure
$\rho$ in $\R^n \times \R^m$ (according to any reasonable definition of ``dimension'')
coincide with the respective dimensions of $\mu$ in $\R^n$ and $\nu$ in $\R^m$. Hence
\thmref{thm_B2.1} in fact provides a  way to construct many examples 
(including fractal ones) of ``mixed type'' measures which nevertheless do admit a Fourier frame.

\subsection{}
We now give the proof of \thmref{thm_B2.1}.

\begin{proof}[Proof of \thmref{thm_B2.1}]
We assume that $\mu, \nu$ are continuous measures (that is, both have no discrete part)
and that each one of them has a
Fourier frame. Let $U\subset \R^n$, $V\subset \R^m$ be countable sets
such that the exponential system $E(U)$ is a frame in $L^2(\mu)$, and
$E(V)$ a frame in $L^2(\nu)$. Hence there exist constants $0<A,B<\infty$ such that 
\begin{equation}
	\label{eq_B1.1.3}
	A  \|g\|^2_{L^2(\mu)}\le \sum_{u\in U}
	|\dotprod{g}{e_{u}}_{L^2(\mu)}|^2 \le B\|g\|^2_{L^2(\mu)}
\end{equation}
and 
\begin{equation}
	\label{eq_B1.1.4}
	A\|h\|^2_{L^2(\nu)}\le \sum_{v\in V}|\dotprod{h}{e_v}_{L^2(\nu)}|^2\le
	B\|h\|^2_{L^2(\nu)} 
\end{equation}
for every $g\in L^2(\mu)$, $h\in L^2(\nu)$.
\par
 Fix a positive integer
$q$ satisfying 
\begin{equation}
	\label{eq_B1.1.5}
	q > \frac{2B}{A}.
\end{equation}
For each $u\in U$, choose a subset $F(u)$ of $V$ consisting of exactly
$q$ elements, in such a way that the sets $F(u)$, $u \in U$, are
disjoint subsets of $V$. This is possible since $\nu$ is a continuous
measure, hence $L^2(\nu)$ is an infinite dimensional space and so
$V$ must be an infinite set. In a completely symmetric way we may
choose, for each $v\in V$, a subset $G(v)$ of $U$ consisting of
exactly $q$ elements, and such that the sets $G(v)$, $v\in V$,
are disjoint subsets of $U$. 
Now define 
\[ \Lambda_\mu := \left\{ (u,v)\, : \, u\in U, \, v\in
	F(u) \right\}, \quad \Lambda_\nu := \left\{ (u,v)\, : \, 
v\in V, \, u\in G(v)\right\}. \]
It would be convenient to know that $\Lambda_{\mu}$ and $\Lambda_{\nu}$ are
disjoint subsets of $U\times V$, which we may assume with no loss of generality
by an appropriate choice of the sets $F(u)$, $G(v)$.
\par
Finally, let
\[ \Lambda := \Lambda_\mu \cup \Lambda_\nu. \] 
We claim that the exponential system $E(\Lambda)$ is a frame in $L^2(\rho)$.
\par
 Indeed, let $f\in
L^2(\rho)$. Then $f$ has a unique decomposition as
\[ f(x,y)d\rho = g(x)d(\mu\times\delta_0)+ h(y)d(\delta_0\times\nu), \]
where $g\in L^2(\mu)$, $h\in L^2(\nu)$. This is due to the definition 
\eqref{eq_B2.1} of the measure $\rho$ (the uniqueness of this
decomposition is true since $\mu, \nu$ have no atom at the origin).
Notice that for any $(u,v)\in \R^n\times \R^m$ we have 
\[
	\dotprod{f}{e_{(u,v)}}_{L^2(\rho)}=\dotprod{g}{e_u}_{L^2(\mu)}+
\dotprod{h}{e_v}_{L^2(\nu)}. \]
\par
For any two complex numbers $a,b$ we have 
\begin{equation}
	\label{eq_B1.1.2}
	|a+b|^2 \ge \tfrac{1}{2}|a|^2-|b|^2, 
\end{equation}
hence 
\[
\begin{aligned}
		\sum\limits_{\lambda\in\Lambda_\mu}
		|\dotprod{f}{e_\lambda}|^2 &= \sum\limits_{u\in
		U}\sum\limits_{v\in F(u)}
		|\dotprod{g}{e_u}+\dotprod{h}{e_v}|^2 
		\ge \sum\limits_{u\in U}\sum\limits_{v\in F(u)} \left(
			\tfrac{1}{2} |\dotprod{g}{e_u}|^2-
		|\dotprod{h}{e_v}|^2
	\right). 
\end{aligned}
\]
Using the fact that the sets $F(u)$ are disjoint subsets of
$V$ with exactly $q$ elements each, this implies that 
\[
\begin{aligned}
		\sum\limits_{\lambda\in \Lambda_\mu}
		|\dotprod{f}{e_\lambda}|^2 &\ge 
		\frac{q}{2}\sum\limits_{u\in U}|\dotprod{g}{e_u}|^2 - 
		\sum\limits_{v\in V} |\dotprod{h}{e_v}|^2  
		\ge \frac{q}{2} A\|g\|^2_{L^2(\mu)} -
		B\|h\|^2_{L^2(\nu)}, 
\end{aligned}
\]
where we have used \eqref{eq_B1.1.3} and \eqref{eq_B1.1.4} in the last
inequality. In a completely symmetric way one can also show that 
\[ 
		\sum\limits_{\lambda\in\Lambda_\nu}
		|\dotprod{f}{e_\lambda}|^2 \ge 
		\frac{q}{2}A \|h\|^2_{L^2(\nu)} -
		B\|g\|^2_{L^2(\mu)}.
\]
Summing the last two inequalities (and using the assumption that
$\Lambda_\mu$ and $\Lambda_\nu$ are disjoint sets) we get
\[
\begin{aligned}
		\sum\limits_{\lambda\in\Lambda}
		|\dotprod{f}{e_\lambda}|^2 &=
		\sum\limits_{\lambda\in\Lambda_\mu} |\dotprod{f}{e_\lambda}|^2 +
		\sum\limits_{\lambda\in\Lambda_\nu} |\dotprod{f}{e_\lambda}|^2\\
		&\ge \left( \frac{q}{2}A-B \right) \left(
			\|g\|^2_{L^2(\mu)}+\|h\|^2_{L^2(\nu)}
		\right) 
		= \left( \frac{q}{2}A-B \right)
		\|f\|^2_{L^2(\rho)}.
\end{aligned}
\]
Due to \eqref{eq_B1.1.5} this provides a lower frame
bound for the system $E(\Lambda)$ in $L^2(\rho)$.
\par
 The upper frame
bound can be obtained in a similar way, by using the inequality 
\[ |a+b|^2\le 2 \left( |a|^2+|b|^2 \right) \]
instead of \eqref{eq_B1.1.2} in the estimates above. Indeed, we have 
\[
\begin{aligned}
		\sum\limits_{\lambda\in\Lambda_\mu}
		|\dotprod{f}{e_\lambda}|^2 &= 
		\sum\limits_{u\in U}\sum\limits_{v\in F(u)}
		|\dotprod{g}{e_u}+\dotprod{h}{e_v}|^2 
		\le 2\sum\limits_{u\in U}\sum\limits_{v\in F(u)}
		\left( |\dotprod{g}{e_u}|^2+|\dotprod{h}{e_v}|^2
		\right) \\
		&\le 2q \sum\limits_{u\in U} |\dotprod{g}{e_u}|^2 +
		2\sum\limits_{v\in V}|\dotprod{h}{e_v}|^2 
		\le  2qB\|g\|^2_{L^2(\mu)} +
		2B\|h\|^2_{L^2(\nu)}.
\end{aligned}
\]
Similarly one can verify that 
\[ \sum\limits_{\lambda\in \Lambda_\nu} |\dotprod{f}{e_\lambda}|^2 \le 
2q B\|h\|^2_{L^2(\nu)} + 2B\|g\|^2_{L^2(\mu)}, \]
and summing the last two inequalities we get the estimate from above 
\[ \sum\limits_{\lambda\in\Lambda}|\dotprod{f}{e_\lambda}|^2 \le
2(q+1)B \|f\|_{L^2(\rho)},\]
which confirms  that the exponential system $E(\Lambda)$ is indeed a frame
for $L^2(\rho)$. 
\end{proof}

\subsection{} 
Although  \thmref{thm_B2.1} allows us to construct many examples of
``mixed type'' measures with a Fourier frame, there exist certain
limitations on the possible dimensions of the components of the
measure  in \eqref{eq_B2.1}. For instance, the theorem does not tell
us whether the measure in $\R^4$ which is the sum of the
$2$-dimensional area measure on $[0,1]^2\times \{0\}^2$ and the
$3$-dimensional area measure on $\{0\}\times [0,1]^3$ has a Fourier
frame. 
\par
However, more general examples are easy to construct by taking the
product $\rho\times \sigma$ of the measure $\rho$ in
\eqref{eq_B2.1} with any measure $\sigma$ in $\R^l$ $(l \ge 1)$
such that also $\sigma$ has a Fourier frame. For it is a general fact
that if $E(\Lambda)$ is a frame in $L^2(\rho)$ and $E(\Gamma)$ a frame
in $L^2(\sigma)$, then $E(\Lambda\times \Gamma)$ is a frame in
$L^2(\rho\times\sigma)$, hence $\rho\times\sigma$ also has a Fourier
frame. 
\par
For example, in this way one can obtain the following extension of
\thmref{thm_A2.1} above: 
\begin{theorem}
	\label{thm_B2.2}
	Consider the measure on $\R^d$ which is the sum of the
	$k$-dimensional area measure on $[0,1]^k\times\{0\}^{d-k}$,
	and the $j$-dimensional area measure on
	$\{0\}^{d-j}\times[0,1]^j$ $(1\le j, k\le d-1)$. Then this
	measure has a Fourier frame. 
\end{theorem}
Indeed, the measure in this result is (after permutation of the
coordinates) of the form $\rho\times\sigma$, where $\rho$ is given by
\eqref{eq_B2.1} and $\mu$, $\nu$ and $\sigma$ are respectively the
$n$, $m$ and $p$-dimensional area measures on $[0,1]^n$, $[0,1]^m$ and 
$[0,1]^p\times \{0\}^q$ for appropriate values of $n$, $m$, $p$ and
$q$. Since in this case each one of $\mu$, $\nu$ and $\sigma$ has
a Fourier frame, \thmref{thm_B2.2} follows.
 (In particular, for $d=4$, $k=2$ and $j=3$ we obtain an affirmative
answer to the question mentioned above.)
% ========================================================================

\section{Pure type phenomena for singular continuous measures}
\label{sec_C2}

The results in the previous section contradict the existence of a general
pure type principle within the class of singular continuous measures.
However, our goal in  the present section is to demonstrate that such principles
can nevertheless be established under some extra assumptions on the components of the measure.
In particular we will prove \thmref{thm_A2.2} above.

\subsection{} 
Let $\mu$ be a positive, finite measure on $\R^d$. Given a real number
$\alpha$, $0\le \alpha\le d$, we consider the following condition:
\begin{equation}
	\label{eq_C2.1.1}
	\liminf_{r \to\infty}\frac{1}{r^{d-\alpha}}\int_{|t|<r}|\widehat{\mu}(t)|^2
	dt >0,
\end{equation}
where 
\[ \widehat{\mu}(t)= \int_{\R^d} e^{-2\pi
i\dotprod{t}{x} }d\mu(x), \quad t\in \R^d \]
is the Fourier transform of the measure $\mu$.
\par
 Intuitively, we think of condition \eqref{eq_C2.1.1} in some sense 
as saying that $\mu$ has
``at least one component of dimension not greater than $\alpha$''.
For example, if $\mu$ is an absolutely continuous measure 
with an $L^2$ density, then \eqref{eq_C2.1.1} is satisfied only for
$\alpha=d$, by the Plancherel identity. 
At the other extreme, by the
classical Wiener's lemma, the condition \eqref{eq_C2.1.1} holds with
$\alpha=0$ if and only if $\mu$ has at least one non-zero atom
(see \cite{Str90a}). 
\par
An example of a singular continuous measure
$\mu$ satisfying \eqref{eq_C2.1.1} for some $\alpha<d$ is given by the
area measure on an open subset of a smooth $k$-dimensional submanifold
in $\R^d$ ($1\le k\le d-1$). In this case the condition
\eqref{eq_C2.1.1} holds for $\alpha=k$ \cite{AH76, Str90a}. 
It was also proved in \cite{Str90a} that \eqref{eq_C2.1.1} is true if
$\mu$ is the $\alpha$-dimensional Hausdorff measure on certain
self-similar fractals of dimension $\alpha$ (in the
latter case $\alpha$ need not be an integer). See also \cite{Str90b, Str93a,
Str93b} for further results in this connection. 
\subsection{}
Recall that a system of vectors $\{f_n\}$ in a Hilbert space $H$ is
called a \define{Bessel system} if there is a positive constant
$C$ such that the inequality 
\begin{equation}
	\label{eq_C2.1.1b}
	\sum_n |\dotprod{f}{f_n}|^2\le C\|f\|^2
\end{equation}
(Bessel's inequality) is satisfied for every $f\in H$. 
\begin{lem}
	\label{lem_C2.1}
	Suppose that a system of exponentials $E(\Lambda)$ constitutes a
	Bessel system in $L^2(\mu)$, where $\mu$ is a measure
	satisfying \eqref{eq_C2.1.1} for some $\alpha$, $0\le \alpha\le d$. Then
	\begin{equation}
		\label{eq_C2.1.2}
		\sup\limits_{x\in\R^d}\#\big(\Lambda\cap B(x,r)\big)
		\le C r^\alpha \quad (r\ge 1)
	\end{equation}
	for a certain constant $C$ which does not depend
	on $r$. 
\end{lem}
Here and below $B(x,r)$ denotes the open ball of radius
$r$ centered at the point $x$.
\par
\lemref{lem_C2.1} should be compared to \cite[Theorem 3.5(a)]{DHSW11}
where the estimate \eqref{eq_C2.1.2} was proved for exponential
Bessel systems $E(\Lambda)$ on certain self-similar measures
$\mu$ of fractal dimension $\alpha$. For such measures which also satisfy
condition
\eqref{eq_C2.1.1} we obtain the latter result as a special case of 
\lemref{lem_C2.1}.

\begin{proof}[Proof of \lemref{lem_C2.1}]
	Assume that $E(\Lambda)$ is a Bessel system in $L^2(\mu)$.
	Then applying Bessel's inequality \eqref{eq_C2.1.1b} to the functions $e_t$
	($t\in\R^d$)  we obtain 
	\begin{equation}
		\label{eq_C2.1.3}
		\sum_{\lambda\in
		\Lambda}|\widehat{\mu}(t-\lambda)|^2\le C_1, \quad
		t\in \R^d
	\end{equation}
	for a certain constant $C_1$. Integrating this inequality over the ball $B(x,2r)$ yields 
\begin{equation}
	\label{eq_C2.1.5}
	\int_{B(x,2r)}\sum_{\lambda\in
	\Lambda}|\widehat{\mu}(t-\lambda)|^2 \, dt\le C_2 \, r^d.
\end{equation}
On the other hand, we have
\[
	\int_{B(x,2r)}\sum_{\lambda\in\Lambda}
	|\widehat{\mu}(t-\lambda)|^2 \, dt = \sum_{\lambda\in \Lambda}
	\int_{B(x-\lambda,2r)} |\widehat{\mu}(t)|^2 \,  dt 
	\ge \sum_{\lambda\in \Lambda\cap B(x,r)}
	\int_{B(x-\lambda,2r)} |\widehat{\mu}(t)|^2 \, dt.
\]
Notice that for each $\lambda\in \Lambda\cap B(x,r)$, the ball 
$B(x-\lambda,2r)$ contains $B(0,r)$. Hence
\[
	\int_{B(x,2r)}\sum_{\lambda\in\Lambda}
	|\widehat{\mu}(t-\lambda)|^2 \, dt 
	\ge \#\big(\Lambda\cap B(x,r)\big)
	\int_{|t|<r}|\widehat{\mu}(t)|^2 \, dt.
\]
Due to the assumption \eqref{eq_C2.1.1}, the integral on the right
hand side of the last inequality is not less than  $C_3 \,
r^{d-\alpha}$ for all sufficiently large $r$, for an appropriate constant
$C_3>0$. Combining this with \eqref{eq_C2.1.5}
we obtain the assertion of the lemma. 
\end{proof}

\subsection{}
Given a set $\Lam \subset \R^d$, one defines its \define{upper
Beurling dimension} to be the infimum  of the numbers $\alpha$ 
for which there exists a constant $C$ such that \eqref{eq_C2.1.2} holds
(see \cite{CKS08}). \lemref{lem_C2.1} thus shows
that if $E(\Lambda)$ is a Bessel system for a measure $\mu$ 
satisfying the condition \eqref{eq_C2.1.1} for some $\alpha$, then
the upper Beurling dimension of $\Lambda$ cannot exceed
$\alpha$. 
\par
If the exponential system $E(\Lambda)$ is not just a Bessel system,
but moreover is a frame in $L^2(\mu)$ for some ``$\alpha$-dimensional'' 
measure $\mu$, then in spirit of the classical Landau's results \cite{Lan67}
one might expect that the upper Beurling
dimension of $\Lambda$ should in fact be equal to $\alpha$.
A result of this type was proved in  \cite[Theorem 3.5(b)]{DHSW11},
where the authors showed that for certain
self-similar measures $\mu$, if $E(\Lambda)$ is a frame in
$L^2(\mu)$ and if the set $\Lambda$ is assumed to enjoy a certain
structure, then the upper Beurling dimension of $\Lambda$ indeed
coincides with the fractal dimension of the measure. 
\par
 However, in general this is not the case. As a simple example one may
 take $\mu$ to be the arc length measure on the segment $[0,1] \times \{0\}$ in
$\R^2$. Then $\mu$ is a one-dimensional measure, and
it is easy to verify that if
$ \Lambda =  \{ (n,2^{|n|}) :  n\in \Z  \} $
then the system $E(\Lambda)$ constitutes an orthonormal basis in $L^2(\mu)$.
Nevertheless,  the estimate 
\[ \sup_{x\in\R^2} \#\big(\Lambda\cap B(x,r)\big) \le C  \log r\]
holds for all sufficiently large $r$, and in particular, the upper
Beurling dimension of $\Lambda$ is zero. In a similar way one can
construct examples of this type in $\R^d$ for any $d\ge 2$.
See also \cite{DHL13} where such an example for fractal measures
on $\R$ was given.

\subsection{}
On the other hand, if we impose extra assumptions on the measure $\mu$,
then we can show that if $E(\Lambda)$ is a frame in $L^2(\mu)$ then
\eqref{eq_C2.1.2} cannot hold for arbitrarily small $\alpha$.
\par
 For a real number $\beta$, $0<\beta\le d$, consider the following condition on the
measure $\mu$:
\begin{equation}
\label{eq_C2.1.7}
\begin{gathered}
\text{there exists a function $\varphi\in L^2(\mu)$, $\|\varphi\|_{L^2(\mu)}>0$, such that} \\
|(\varphi   \mu)^{\land}   (t)| \le C \, 	|t|^{-\beta/2}, \quad  t\in \R^{d}
\end{gathered}
\end{equation}
for some positive constant $C$. Notice that $\varphi   \mu$ is a
(non-zero) finite, complex measure on $\R^d$, since $\varphi \in L^2(\mu)$ and $\mu$ is a
finite measure. 
\par
It is known  that (at least, for non-negative $\varphi$) the estimate in \eqref{eq_C2.1.7} implies that
the measure $\varphi  \mu$ cannot charge any compact set of
Hausdorff dimension smaller than $\beta$ (see for example \cite[Sections 2.5, 3.5, 3.6]{Mat15}).
This means, intuitively, that $\mu$ has ``at least one
component of dimension not less than $\beta$''.
\par
 A well-known example of
a singular measure $\mu$ satisfying \eqref{eq_C2.1.7} for some
$\beta>0$ is the area measure on the unit sphere in $\R^d$,
or, more generally, on an open subset $\Omega$ of a smooth hypersurface
in $\R^d$ with a point of non-zero Gaussian curvature.
In this case one may take $\varphi$ to be any smooth
function supported by a sufficiently small neighborhood of a point $x_0 \in \Omega$
where the Gaussian curvature does not vanish, and the estimate in \eqref{eq_C2.1.7} then holds with
$\beta=d-1$ (see e.g.\ \cite[Section VIII.3]{Ste93}).
\par
It should be mentioned that the theory of \define{Salem sets} is concerned
with various constructions of sets in $\R^d$ of Hausdorff dimension 
$\beta$ which support a measure $\mu$ satisfying \eqref{eq_C2.1.7}. See 
\cite[Section 3.6]{Mat15} and the references mentioned there.

\begin{lem}
\label{lem_C2.2}
	Suppose that a system of exponentials $E(\Lambda)$ is a
	frame in $L^2(\mu)$, where $\mu$ is a measure satisfying
	\eqref{eq_C2.1.7} for some $\beta$, $0<\beta\le d$. If there is $\alpha$ such that 
	\begin{equation}
		\#\big(\Lambda\cap B(0,r)\big) \le C  r^{\alpha}
		\quad (r\ge 1)
		\label{eq_C2.2.1}
	\end{equation}
	for some constant $C$, then we must have 
	\begin{equation}
		\label{eq_C2.2.4}
		\alpha \ge \left(
		\frac{1}{\beta}+\frac{1}{d} \right)^{-1}.
	\end{equation}
\end{lem}
\begin{proof}
	By assumption there exists a function $\varphi \in L^2(\mu)$, $\|\varphi
	\|_{L^2(\mu)}>0$, such that the (non-zero, finite, complex) measure
	$\nu:=\varphi  \mu$ satisfies 
	\begin{equation}
		\label{eq_C2.2.2}
		|\hat{\nu}(t)|\le C_1 \, |t|^{-\beta/2}, \quad t\in \R^d.
	\end{equation}
	We also assume that the system $E(\Lambda)$ is a frame in $L^2(\mu)$.
	Then applying the lower frame inequality to the functions
	$\overline{\varphi} \cdot e_t$ $(t\in \R^d)$ yields 
	\begin{equation}
		\label{eq_C2.2.3}
		\sum_{\lambda\in \Lambda} |\hat{\nu}(t-\lambda)|^2 \ge
		C_2 > 0 , \quad t\in \R^d.
	\end{equation}
\par
	Suppose now to the contrary that \eqref{eq_C2.2.1} does hold for some
	$\alpha$ such that
	\begin{equation}
		\label{eq_C2.2.5}
		\alpha< \left( \frac{1}{\beta}+\frac{1}{d} \right)^{-1}.
	\end{equation}
	Then this allows us to choose a real number $\gamma$ satisfying 
	\begin{equation}
		\label{eq_C2.2.6}
		\frac{\alpha}{\beta}<\gamma<1-\frac{\alpha}{d}. 
	\end{equation}
\par
	Given a large number $R$ we define $T:= R^\gamma$ and consider
	the union of balls $B(\lambda,T)$ of radius $T$ centered at
	the points $\lambda \in \Lambda\cap B(0,R)$. Then the
	$d$-dimensional volume of the union of these balls is not greater than 
	\[ \#\big(\Lambda \cap B(0,R)\big)\cdot |B(0,T)| \le C_3 \,
		R^\alpha \,  T^d = C_3 \, R^{\alpha+\gamma 
	d}= o(R^d), \]
	due to \eqref{eq_C2.2.1} and \eqref{eq_C2.2.6}. Hence if
	$R$ is sufficiently large then these balls cannot cover 
	$B(0,R/2)$. So we may find a point $t_0\in B(0,R/2)$ such
	that $|t_0-\lambda|\ge T$, $\lambda\in \Lambda\cap B(0,R)$. We
	will show that the inequality \eqref{eq_C2.2.3} is then violated at
	this point $t_0$.
\par
 Indeed, we have 
	\[ \sum\limits_{\lambda\in\Lambda,\; |\lambda|<R}
		|\hat{\nu}(t_0-\lambda)|^2 \le \#\big(\Lambda\cap
		B(0,R)\big) \cdot \sup_{|t| \geq T} |\hat{\nu}(t)|^2 \le C_4 \,
	R^{\alpha} \, T^{-\beta}=C_4\, R^{\alpha
	-\gamma \beta},\]
	according to \eqref{eq_C2.2.1} and \eqref{eq_C2.2.2}. We also
	have 
	\[
		\begin{aligned}
		\sum\limits_{\lambda\in \Lambda, \; |\lambda|\ge R}
		|\widehat{\nu}(t_0-\lambda)|^2 &=
		\sum\limits_{k=0}^{\infty}\sum\limits_{\quad 2^k R\le
		|\lambda|<2^{k+1}R} 
		|\widehat{\nu}(t_0-\lambda)|^2  \\ 
		&\le \sum_{k=0}^{\infty} \#\big(\Lambda\cap	
		B(0,2^{k+1} R)\big) \cdot
		\sup_{|\lambda|\ge 2^k R}
		|\widehat{\nu}(t_0-\lambda)|^2 \\ 
		&\le C_5 \sum_{k=0}^{\infty} \left( 2^{k+1}R \right)^{\alpha} 
		\left( 2^{k-1} R \right)^{-\beta} = C_{6} \, R^{\alpha-\beta},
		\end{aligned}
	\]
	where we have used \eqref{eq_C2.2.1}, \eqref{eq_C2.2.2} 
	and the fact that \eqref{eq_C2.2.5} implies
	that $\alpha<\beta$. It follows that 
	\begin{equation}
		\label{eq_C2.2.7}
		\sum\limits_{\lambda\in \Lambda}
		|\widehat{\nu}(t_0-\lambda)|^2 \le C_4 \,
		R^{\alpha-\gamma  \beta} + C_6 \, R^{\alpha-\beta}.
	\end{equation}
	Due to \eqref{eq_C2.2.6} we have $\alpha-\gamma  \beta<0$, hence
	the right hand side of \eqref{eq_C2.2.7} can be made as small as we wish
	provided that $R$ is chosen sufficiently large. But this contradicts 
	\eqref{eq_C2.2.3}, and so the lemma is proved.
\end{proof}

\subsection{} 
We can now combine the previous lemmas to obtain the main result of
this section, which establishes a pure type principle for singular continuous
measures. 
\begin{theorem}
	\label{thm_C3.1}
	Let $\mu$ be a measure on $\R^d$ which satisfies both conditions
	\eqref{eq_C2.1.1} and \eqref{eq_C2.1.7} for some
	$0<\alpha,\beta\le d$. If 
	\begin{equation}
		\label{eq_C3.1.1}
		\frac{1}{\alpha}-\frac{1}{\beta}>\frac{1}{d}
	\end{equation}
	then $L^2(\mu)$ does not admit a Fourier frame.
\end{theorem}
We view this result as a pure type principle, since the
assumptions mean, intuitively, that $\mu$ has both a component of
dimension not greater than $\alpha$, and (in a certain strong sense) a
component of dimension not less than $\beta$, and the conclusion is
that $\mu$ cannot have a Fourier
frame if $\alpha$ and $\beta$ are sufficiently far apart. 
\begin{proof}[Proof of \thmref{thm_C3.1}]
	Suppose to the contrary that $L^2(\mu)$ does have a frame $E(\Lambda)$.
	Then by \lemref{lem_C2.1} the set $\Lambda$ must satisfy \eqref{eq_C2.1.2}. 
	In turn, \lemref{lem_C2.2} implies that the numbers $\alpha$, $\beta$ must satisfy 
	the relation \eqref{eq_C2.2.4}. However this contradicts \eqref{eq_C3.1.1}.
\end{proof}
As an application of \thmref{thm_C3.1}, we can conclude the following result: 
\begin{theorem}
	\label{thm_C3.2}
	Let $\mu$ be a measure on $\R^d$ $(d\ge 3)$ which is the sum
	of the area measure on an open subset of a smooth
	$k$-dimensional submanifold, where 
	\begin{equation}
		\label{eq_C3.2.1}
		1\le k\le \left\lfloor \frac{d-1}{2} \right\rfloor,
	\end{equation}
	and the area measure on an open subset of a smooth
	hypersurface with a point of non-zero Gaussian curvature.
	Then $L^2(\mu)$ does not admit a Fourier frame. 
\end{theorem}
Here $\left\lfloor x \right\rfloor$ denotes the greatest integer which is less
than or equal to a real number $x$.

\begin{proof}[Proof of \thmref{thm_C3.2}]
The assumptions imply that the measure $\mu$ satisfies condition
\eqref{eq_C2.1.1} with $\alpha=k$ (see \cite[Theorem 5.5]{Str90a}) and
condition \eqref{eq_C2.1.7} with $\beta=d-1$ (see \cite[Section
VIII.3.1]{Ste93}). Moreover, it follows from \eqref{eq_C3.2.1} 
that condition \eqref{eq_C3.1.1} is satisfied for these values of
$\alpha$ and $\beta$. Hence the assertion follows from \thmref{thm_C3.1}.
\end{proof}

 In the special case when $k=1$ in \thmref{thm_C3.2} we obtain \thmref{thm_A2.2}.

% ========================================================================

\section{Fourier frames for surface measures}
\label{sec_D1}

In the previous section we proved  that if
a measure $\mu$ in $\R^d$ is the sum of the area measure
on a $k$-dimensional submanifold (at least, for $k$ not too large)
and the area measure on a hypersurface with a point of non-zero Gaussian curvature,
then $\mu$ does not have a Fourier frame. 
Our main goal in the present section is to show that in many examples, a
Fourier frame does exist for each  ``pure'' component of such a measure separately.

\subsection{} 
We will consider a measure $\mu$ in $\R^d$ which is the area measure on a $k$-dimensional
submanifold defined as a \define{graph}. Let $\varphi: U\to \R^{d-k}$ $(1\le k\le d-1)$
be a smooth function defined on an open set $U\subset \R^k$. Let $E$ be a compact subset
of $U$, and let 
\begin{equation}
	\label{eq_D1.1.1}
	\Gamma = \Gamma(\varphi,E)=\left\{ (x,\varphi(x))\, : \, x\in E \right\}
\end{equation}
be the \define{graph of $\varphi$ over $E$}. Then $\Gamma$ is a compact subset of a
$k$-dimensional submanifold in $\R^d=\R^k\times \R^{d-k}$, and so it admits an area
measure $\mu$ induced from its embedding in $\R^k\times \R^{d-k}$. The measure $\mu$ can
be defined by the requirement that 
\begin{equation}
	\label{eq_D1.1.2}
	\int_{\Gamma}f(x,y)d\mu(x,y) = \int_{E} f (x,\varphi(x) )
	w(x)dx
\end{equation}
for every continuous function $f$ on $\Gamma$, where $w(x) > 0$ is a certain 
smooth weight function which depends on the function $\varphi$. Actually, it will be sufficient for us to
assume that $w(x)$ is an arbitrary continuous, strictly positive function on $E$. 
\begin{theorem}
	\label{thm_D1.1}
	The measure $\mu$ on $\R^d$ defined by \eqref{eq_D1.1.2} has a Fourier frame. 
\end{theorem}

\begin{proof}
In fact, it will be easy to verify that if $\delta>0$ is sufficiently small and if 
\begin{equation}
	\label{eq_D1.1.3}
	\Lambda =(\delta \Z)^k\times \{0\}^{d-k},
\end{equation}
then the system $E(\Lambda)$ is a frame in $L^2(\mu)$. Indeed, fix $\delta>0$ small enough
such that $E$ is contained in the $k$-dimensional cube 
\[ I^k_\delta := \Big[ -\frac{1}{2\delta}, \, \frac{1}{2\delta} \Big]^k, \]
and let $\Lambda$ be defined by \eqref{eq_D1.1.3}. Given $f\in L^2(\mu)$, we have 
\begin{align}
	\sum\limits_{\lambda\in\Lambda}|\dotprod{f}{e_\lambda}_{L^2(\mu)}|^2 &=  
	\sum\limits_{m\in\Z^k} \left| \int_E f\big(x,\varphi(x)\big) e^{-2\pi
	i\dotprod{\delta m}{x}} w(x)dx \right|^2 \nonumber \\
	\label{eq_D1.1.5}
	&=  C  \int_E \left| f\big(x,\varphi(x)\big)\right|^2 w(x)^2 dx,
\end{align}
where the last equality is true with a certain constant $C = C(k, \delta)$ since the system of
exponentials $E (\delta\Z^k )$ forms an orthogonal basis in the space $L^2(I^k_\delta)$ with respect
to the Lebesgue measure, and since $E$ is contained in $I_\delta^k$. On the other hand, we have 
\begin{equation}
	\label{eq_D1.1.6}
	\|f\|^2_{L^2(\mu)} = \int_E \left|f(x,\varphi(x)\big)\right|^2 w(x)dx.
\end{equation}
Since $w(x)$ is a continuous, strictly positive function on the compact set $E$, it is
bounded from above and from below on $E$. Hence the ratio between \eqref{eq_D1.1.5} and
\eqref{eq_D1.1.6} is also bounded from above and from below by certain positive constants not depending
on $f$. This confirms that the system $E(\Lambda)$ is indeed a frame in $L^2(\mu)$. 
\end{proof}

\subsection{}
In the case when $k=d-1$, the function $\varphi$ is scalar-valued, and $\Gamma$ is then a
subset of a hypersurface in $\R^d$. We remark that the requirement in
\thmref{thm_C3.2} that the hypersurface
has a point of non-zero Gaussian curvature is easy to express in terms of the function
$\varphi$: this is the case if and only if the $(d-1)\times(d-1)$ matrix given by
\[ \left( \frac{\partial^2\varphi}{\partial x_j\partial x_k} \right)(x) \]
is invertible at some point $x\in U$.

% ========================================================================

\section{Open problems}

Finally we mention some questions concerning possible extensions of our results.

\subsection{}
Notice that
\thmref{thm_B2.1} allows us to construct ``mixed type'' measures with a Fourier frame only
in dimensions $2$ and higher. It would be interesting to know whether examples of this type
can be found on $\R$ as well. Specifically: can one construct examples of two self-similar Cantor measures
$\mu$ and $\nu$ on $\R$ of different dimensions, such that their sum $\mu+\nu$ is a measure
with a Fourier frame? 
\par
Here the main point in the proof of \thmref{thm_B2.1} may still be
useful:  it would be enough to show that given any positive number $q$ there is a set 
$\Lambda_\mu$,  such that the exponential system $E(\Lambda_\mu)$ is a frame in 
$L^2(\mu)$ with lower frame bound not less than $q$,
and at the same time $E(\Lambda_\mu)$ constitutes a Bessel system in $L^2(\nu)$ with Bessel
constant bounded from above independently of $q$; and similarly, to show that there
is also a set $\Lambda_\nu$ such that the system $E(\Lambda_\nu)$ has the same properties 
but with the roles of $\mu$ and $\nu$ interchanged. Indeed, in this case  the proof of \thmref{thm_B2.1} shows
that a frame $E(\Lambda)$ for $L^2(\mu+\nu)$ can
be obtained by choosing $q$ sufficiently large and taking $\Lambda=\Lambda_\mu\cup \Lambda_\nu$. 
\subsection{}
It would also be interesting to understand to what extent \thmref{thm_C3.1} is sharp. Can the
requirement \eqref{eq_C3.1.1} be relaxed to $\alpha<\beta$\,?
Analogously, can one relax the requirement \eqref{eq_C3.2.1} in \thmref{thm_C3.2}  to
$1\le k\le d-2$\,? 
\par
In the same spirit, one may ask whether it is possible to get a version of
\thmref{thm_C3.2} for two submanifolds in  $\R^d$ of different dimensions, but such 
that not necessarily one of the dimensions is equal to $d-1$. What geometric condition would replace the
curvature requirement in such a result? 
(We know  that some condition of this sort is necessary, due to \thmref{thm_B2.2}.)

\subsection{}
Another question is concerned with the assumption in \thmref{thm_D1.1} that the submanifold is
globally contained in the graph of some function. Can this assumption be removed?
In particular, we do not know whether the area measure on the unit sphere 
\[ S^{d-1}=  \{  x\in \R^d \, : \, |x|=1  \} \]
has a Fourier frame $(d \geq 2)$. 
(\thmref{thm_D1.1} only tells us that the restriction of the area measure to a sufficiently small \define{spherical
cap} admits such a frame.)

% ========================================================================

\end{document}